\documentclass[11pt]{article}

\usepackage{graphicx,tikz,color,url}
\usepackage{amsmath,amssymb}
\usetikzlibrary{arrows}
\usepackage[lined,boxed]{algorithm2e}

\usepackage{hyperref}

\newtheorem{theorem}{Theorem}[section]
\newtheorem{definition}[theorem]{Definition}

\newtheorem{proposition}[theorem]{Proposition}
\newtheorem{observation}[theorem]{Observation}

\newtheorem{corollary}[theorem]{Corollary}
\newtheorem{remark}[theorem]{Remark}

\newtheorem{problem}[theorem]{Problem}

\newtheorem{claim}{Claim}

\newcommand{\proof}{\noindent{\bf Proof.\ }}
\newcommand{\qed}{\hfill $\square$\bigskip}

\newcommand{\gsmb}{\gamma_{\rm SMB}}

\newcommand{\smallqed}{{\tiny ($\Box$)}}

\newcommand{\cH}{{\cal H}}
\newcommand{\cF}{{\cal F}}
\newcommand{\cC}{{\cal C}}
\newcommand{\cS}{{\cal S}}

\newcommand{\cN}{{\cal N}}

\newcommand{\cA}{{\cal A}}
\newcommand{\cX}{{\cal X}}
\newcommand{\cL}{{\cal L}}

\begin{document}

\title{Criticality for Maker-Breaker domination games with predomination}
\date{}

\author{Csilla Bujt\'as$^{a,b,}$\thanks{Email: \texttt{csilla.bujtas@fmf.uni-lj.si}} 
	\hspace{1cm} %\and 
    Pakanun Dokyeesun$^{b,}$\thanks{Email: \texttt{papakanun@gmail.com}} \\ \\ %\and 
Sandi Klav\v zar$^{a,b,c,}$\thanks{Email: \texttt{sandi.klavzar@fmf.uni-lj.si}}
       \hspace{1cm} %\and 
       Milo\v{s} Stojakovi\'{c} $^{d,}$\thanks{Email: \texttt{milos.stojakovic@dmi.uns.ac.rs}}  
}
\maketitle

\begin{center}
	$^a$ Faculty of Mathematics and Physics, University of Ljubljana, Slovenia\\
	\medskip

	$^b$ Institute of Mathematics, Physics and Mechanics, Ljubljana, Slovenia\\
	\medskip
	
	$^c$ Faculty of Natural Sciences and Mathematics, University of Maribor, Slovenia\\
	\medskip

        $^d$ Department of Mathematics and Informatics, Faculty of Sciences, \\University of Novi Sad, Serbia\\
\medskip

\end{center}
\begin{abstract}
A predominated graph is a pair $(G,D)$, where $G$ is a graph and the vertices in $D\subseteq V(G)$ are considered already dominated. Maker-Breaker domination game critical (MBD critical) predominated graphs are introduced as the predominated graphs $(G,D)$ on which Staller wins the game, but Dominator wins on $(G, D \cup \{v\})$ for every vertex $v \in V(G) \setminus D$. 

Tools are developed for handling the Maker-Breaker domination game on trees which lead to a characterization of Staller-win predominated trees. MBD critical predominated trees are characterized and an algorithm is designed which verifies in linear time whether a given predominated tree is MBD critical. A large class of MBD critical predominated cacti is presented and Maker-Breaker critical hypergraphs constructed. 
\end{abstract}
\noindent
{\bf Keywords:} domination game; Maker-Breaker game; Maker-Breaker domination game; predomination; hypergraph \\

\noindent
{\bf AMS Subj.\ Class.\ (2020)}: 05C57, 05C65, 05C69

%%%%%%%%%%%%%%%%%%%%%%%%%%%%%%
\section{Introduction}
%%%%%%%%%%%%%%%%%%%%%%%%%%%%%%

%%
\noindent\textbf{Positional games.}
Positional games form a specific and well studied subclass of combinatorial games, that includes popular recreational games like Tic-Tac-Toe, Hex and Sim. They have been explored in depth in two books~\cite{beck-book, hefetz-2014}. Structurally, a positional game is a hypergraph $\cH=(V,E)$, where $V$ is a finite set representing the \emph{board} of the game, and $E \subseteq 2^V \setminus \{\emptyset\}$ is a family of sets that we refer to as the \emph{winning sets}. The game is played by two players who alternately claim the unclaimed elements of the board, until all of them are claimed.

When it comes to the game's outcome, there are several standard conventions. Here, we highlight two of the most significant ones, which were also the first to be introduced and studied, see seminal and pioneering papers of Hales and Jewett~\cite{hales-1963} and Erd\H{o}s and Selfridge~\cite{erdos-1973}. In a \emph{strong game}, the first player to claim all elements of a winning set is the winner, and if the game finishes without a winner a draw is declared. In a \emph{Maker-Breaker game}, the players are called Maker and Breaker, and there are only two possible outcomes of the game; Maker's goal is to claim all elements of a winning set, while Breaker wins otherwise, i.e.\ if Breaker claims an element in every winning set. If not specified otherwise, we will assume that Maker starts the game.

One of the central general questions about positional games is the monotonicity with respect to the edge set of the game hypergraph. It was observed early on by Beck (see, e.g.,~\cite[Section 5]{beck-book}) that adding an extra winning set to a strong game can change the outcome in either direction -- a draw can turn into a first player's win, but also a first player's win can turn into a draw. Beck referred to this phenomenon as the `Extra Set Paradox', highlighting that this lack of monotonicity of strong games might initially seem surprising, and in sharp contrast with the well-known monotonicity of Maker-Breaker games. Indeed, it is both folklore and an observation that adding an extra winning set to a Maker's win game always results in a Maker's win game, as Maker can simply win by applying the exact same winning strategy as before.

Once this general monotonicity of Maker-Breaker games is established, a standard topic of interest in extremal combinatorics is the study of critical hypergraphs -- those games $(V,E)$ that are a win for Maker, but where removing any edge $e\in E$ results in a game $(V,E\setminus \{e\})$ that is a win for Breaker. This topic is also central to our investigation in the present paper.
\medskip

\textbf{Maker-Breaker domination game.}
A substantial amount of research on positional games have been conducted on games on graphs, where the board consists of either the edge set or the vertex set of a graph. Games played on the edges of graphs are more prevalent in the literature; see~\cite{hefetz-2014} for an overview. This is partly due to the broader range of games that can be studied in this setting, including those on complete graphs, first analyzed by Chvátal and Erdős~\cite{chvatal-1978}.

When it comes to Maker-Breaker games on vertices, playing on complete graphs is meaningless, as there is no structure to play for -- every vertex-induced graph is itself a complete graph. Consequently, there are considerably fewer games for which the outcome is known (or it is efficiently computable, given the base graph on which the game is played).

One notable exception is the domination game, which naturally arises from the well-studied concept of domination on graphs. Given a graph $G$, the board of the game is the vertex set $V(G)$, while the winning sets are the closed neighborhoods of all vertices in $G$. Maker's goal is to claim a closed neighborhood, and Breaker wants to prevent that, which is equivalent to claiming a dominating set of $G$. This game was first analyzed in~\cite{duchene-2020}, and its various aspects have been further looked at in~\cite{bagan-2025, bujtas-2024, bujtas-2023, gledel-2019}.
\medskip

\textbf{Classical domination game.} 
The Maker-Breaker domination game can be considered a ``younger sibling'' of a well-established combinatorial game on graphs, widely studied in the literature under the same name: the domination game, let us call it here the classical domination game in order to distinguish it from the (Maker-Breaker) domination game. 

The two games are closely related in several ways. Among these, it inherits the game hypergraph structure -- it also revolves around the concept of a dominating set.
In this game, two players, Staller and Dominator, alternately claim unclaimed vertices of a given graph $G$, \emph{jointly} building a dominating set. Each claimed vertex must expand the set of dominated vertices, and the game concludes once the vertices claimed by both players form a dominating set. Staller aims to prolong the game as much as possible, while Dominator seeks to minimize its duration. 

A central question is determining the game's length under optimal play for a given graph $G$. The classical domination game was introduced in~\cite{bresar-2010}, for a comprehensive overview of this and related problems, see the book~\cite{book-2021}. Because in this article we focus on the criticality of games, we emphasize that graphs critical for the classical domination game have been investigated in~\cite{bujtas-2015, dorbec-2019, xu-2018}, and graphs critical for the total domination game in~\cite{charoensitthichai-2020, henning-2018a, henning-2018b, worawannotai-2024}. For more detail about total domination game see~\cite{henning-2015} and ~\cite[Sections 2.9 and 3.6]{book-2021}. It is also important to point to~\cite{bujtas-2021} for a closely related concept of perfect graphs for domination games. 

A natural generalization of this game introduces a \emph{predominated} set of vertices $D\subseteq V(G)$, where the game play remains unchanged except that the vertices in $D$ are considered already dominated -- the game on $(G,D)$ concludes once all vertices in $V(G)\setminus D$ are dominated. The concept of predomination allows for the study of \emph{critical} pairs $(G,D)$, where the duration of the game on $(G,D)$ is strictly longer than on $(G, D\cup\{x\})$, for all $x\in V(G)\setminus D$. The effect of predomination on the total domination game and on the connected domination game was respectively investigated in~\cite{irsic-2019} and in~\cite{bujtas-2022, irsic-2022}. 
\medskip

\textbf{Our concepts.} We now return to the Maker-Breaker domination game, introducing and studying predomination in an analogue way. To align with the standard terminology established for the classical domination game, we will refer to Closed-Neighborhood-Maker as Staller and Closed-Neighborhood-Breaker as Dominator. This also helps to avoid potential confusion, as Closed-Neighborhood-Maker is simultaneously Dominating-Set-Breaker.

In a Maker-Breaker domination game played on $(G,D)$, Dominator wins if the vertices he claims dominate all vertices in $V(G) \setminus D$, while Staller wins if she manages to claim all the vertices in the closed neighborhood $N[v]$ of some vertex $v \notin D$. Note that the vertices in $D$ are part of the game board, they can be played during the game and may belong to the closed neighborhoods of vertices outside $D$. Therefore, a Maker-Breaker domination game on $(G,D)$ in general differs from the game played on on $G-D$. Also, note that choosing $D = \emptyset$ is a valid option. Throughout this paper, we assume that Staller starts the game.

A predominated graph $(G,D)$ is called \emph{Maker-Breaker domination game critical} ({\em MBD critical}, or just {\em critical}) if Staller wins the game on $(G,D)$, but Dominator wins on $(G, D \cup \{v\})$ for every vertex $v \in V(G) \setminus D$. Furthermore, note that if $(G,D)$ is MBD critical and $D' \subsetneq  D$, then $(G,D')$ is not MBD critical. Similarly, if $D\subsetneq D''$, then $(G,D'')$ is also not MBD critical. If Staller wins the game on $(G,D)$, then $D$ can be extended to a set $D'\supseteq D$ such that $(G,D')$ is a critical predominated graph.  
\medskip

{\textbf{Main results.} We develop tools for handling the domination game with predomination on trees, which enable us to give a characterization of Staller-win predominated trees in Theorem~\ref{thm:tree-St-wins}.  Furthermore, we characterize all MBD critical predominated trees in Theorem~\ref{thm:critical-trees}, and we show that the necessary and sufficient conditions for a given predominated tree to be MBD critical are verifiable in linear time in Theorem~\ref{thm:critical-tree-linear}. Then we turn to cacti and construct in Theorem~\ref{thm:cC} a large class of (atomic) MBD critical predominated cacti. The latter approach is extended in Theorem~\ref{thm:cA} to obtain a wider class of (atomic) MBD critical predominated graphs. Finally, in Proposition~\ref{prop:cactus-hg}, Maker-Breaker critical hypergraphs are constructed. 

\medskip

\textbf{Paper organization.} 
In Section~\ref{sec:critical-hgs}, we formally and comprehensively present the Maker-Breaker game critical hypergraphs, while Section~\ref{sec:MBD-critical-graphs} focuses on the MBD critical graphs, along with some of their general properties. Section~\ref{sec:MBD-critical-trees} is dedicated to characterizing MBD critical trees, whereas Section~\ref{sec:MBD-critical-cactuses} analyzes MBD critical cactuses. Finally, in Section~\ref{sec:families-MB-critical-hgs}, we outline potential directions for future work, including a more detailed exploration of three promising avenues: the search for additional families of MBD critical graphs, the concept of criticality for Dominator, and the general study of Maker-Breaker critical hypergraphs.

%%%%%%%%%%%%%%%%%%%%%%%%%%%%%%
\subsection{Preliminaries}
%%%%%%%%%%%%%%%%%%%%%%%%%%%%%%

A hypergraph $\cH'=(V', E')$ is a \emph{subhypergraph} of $\cH=(V,E)$, if $V' \subseteq V$ and $E' \subseteq E$. In a hypergraph $\cH$, the {\em degree} of a vertex is the number of incident edges, in particular, a vertex of degree $0$ is an {\em isolated} vertex.   A hypergraph is \emph{$k$-uniform} if every hyperedge $e \in E$ contains exactly $k$ vertices.

Given a graph $G$, its \emph{closed neighborhood hypergraph} $\cN(G)$ is defined on the vertex set $V(G)$ with hyperedges corresponding to the closed neighborhoods $N_G[v]$ for every $v \in V(G)$. For a predominated graph $(G,D)$, the hypergraph $\cN(G,D)$ is defined on the vertex set $V(G)$ and contains the closed neighborhoods of the non-predominated vertices:
$$ E(\cN(G,D))= \{N_G[v] : v \in V(G) \setminus D \}. $$

%%%%%%%%%%%%%%%%%%%%%%%%%%%%%%%%%%%%%%%%

\section{Maker-Breaker game critical hypergraphs}
\label{sec:critical-hgs}

Here, we define criticality for Maker-Breaker games in general.  We impose no restrictions on the game hypergraph, and in particular we allow the presence of isolated vertices in $V$. 

\begin{definition}
   A hypergraph $\cH=(V,E)$ is \emph{Maker-Breaker critical} if Maker wins in the Maker-Breaker game on $\cH$ but, for every $e \in E$, Breaker wins the game on $\cH -e$.
\end{definition}
\begin{proposition}  \label{prop:critical-hgs-properties}
For every hypergraph $\cH=(V,E)$, the following statements hold.
\begin{itemize}
\item[$(i)$] Maker wins the game on $\cH$ if and only if $\cH$ contains a Maker-Breaker critical subhypergraph.
\item[$(ii)$] Let $V_0 \subseteq V$ be a set of isolated vertices in $\cH$. Then, $\cH$ is Maker-Breaker critical if and only if $\cH-V_0$ is Maker-Breaker critical.
\item[$(iii)$] If $\cH$ is a disconnected Maker-Breaker critical hypergraph, then all but one components of $\cH$ are isolated vertices.
 \end{itemize}  
\end{proposition}
\begin{proof}
  $(i)$ Recall that adding new winning sets to the hypergraph is never disadvantageous for Maker. If $\cH$ contains a Maker-Breaker critical subhypergraph $\cH'$, then Maker wins on $\cH'$ and therefore wins on $\cH$. For the other direction, assume that Maker wins on $\cH$. If $\cH$ is Maker-Breaker critical, then we are done. If this is not the case, there is a hyperedge $e \in E$ such that Maker wins on $\cH-e$. We repeat this operation while it is possible; i.e., we remove an edge such that Maker wins on the hypergraph obtained. As Maker does not win if $E'=\emptyset$, we will arrive at a subhypergraph $\cH'$ that is Maker-Breaker critical. 

  For $(ii)$, we note that the presence or absence of isolated vertices in $\cH$ does not affect the outcome of the Maker-Breaker game. 
  
  To prove $(iii)$, suppose that $\cH$ is disconnected and consists of components $\cH_1, \dots, \cH_\ell$. Breaker as second can win the game on $\cH$ if each component $\cH_i$, for $i \in [\ell]$, is a Breaker-win graph. Indeed, if Breaker replies to every move of Maker in $\cH_i$ by playing a vertex from the same component according to a winning strategy on $\cH_i$, Maker can never claim a winning set. (Breaker may choose an arbitrary vertex from $\cH$ if there is no unplayed vertex in $\cH_i$ after the move of Maker.) Consequently, if Maker can win the game on $\cH$, then there is at least one component, say $\cH_1$, such that Maker wins on $\cH_1$. Removing all winning sets outside $\cH_1$ leaves a hypergraph on which Maker wins. Thus, if $\cH$ is Maker-Breaker critical, there are no hyperedges outside $\cH_1$, and equivalently, the other components are all isolated vertices.  
\end{proof}
\qed

By Proposition~\ref{prop:critical-hgs-properties}~(ii), we may restrict our attention to Maker-Breaker critical hypergraphs without isolated vertices that we call \emph{atomic Maker-Breaker critical hypergraphs}. As a consequence of Proposition~\ref{prop:critical-hgs-properties}, we may also state the following properties.
\begin{corollary} \enskip
   \begin{itemize}
       \item[(i)] Maker wins the game on a hypergraph $\cH$ if and only if $\cH$ contains an atomic Maker-Breaker critical subhypergraph.
\item[$(ii)$] If $\cH$ is an atomic Maker-Breaker critical hypergraph, then $\cH$ is connected.
   \end{itemize} 
\end{corollary}

In Section~\ref{sec:families-MB-critical-hgs}, we present several further statements and examples for Maker-Breaker critical hypergraphs.
%%%%%%%%%%%%%%%%%%%%%%%%%%%%%%%%%%%%%%

\section{Maker-Breaker domination game critical graphs} \label{sec:MBD-critical-graphs}

It was first observed in~\cite{duchene-2020} (see also~\cite[Observation 1.3]{bujtas-2024}) that a Maker-Breaker domination game on a graph $G$ can be considered as a Maker-Breaker game on $\cN(G)$ where Staller plays as Maker and Dominator plays as Breaker. Equivalently, Staller's goal is to claim the entire closed neighborhood $N_G[v]$ for a vertex $v$. If she is successful in this goal, it will make it impossible for Dominator to dominate $v$, and Staller wins.  In a predominated graph $(G,D)$, the closed neighborhoods of the predominated vertices are no longer winning sets for Staller, and we obtain the following interrelation between the games.
\begin{observation} \label{obs:closed-neigh-hg-predomination}
    A Maker-Breaker domination game on the predominated graph $(G,D)$ corresponds to the Maker-Breaker game on the hypergraph $\cN(G,D)$, where Staller is Maker and Dominator is Breaker.
\end{observation}

\begin{proposition} \label{prop:MBDcritical}
The following statements hold for every predominated graph $(G,D)$.
\label{propositions}
\begin{itemize}
    \item[$(i)$] If $e$ is an edge between two vertices of $D$, then Staller can win on $(G,D)$ if and only if she wins on $(G-e,D)$. Further, $(G,D)$ is MBD critical if and only if $(G-e,D)$ is MBD critical.
    \item[$(ii)$] If a predominated vertex $u \in D$ is isolated in $G$, then Staller can win on $(G,D)$ if and only if she wins on $(G-u,D\setminus \{u\})$. Further, $(G,D)$ is MBD critical if and only if $(G-u,D \setminus \{u\})$ is MBD critical.  
     \item[$(iii)$] If $G$ is disconnected and $(G,D)$ is MBD critical, then all vertices in $V(G) \setminus D$ belong to the same component of $G$.
\end{itemize}    
\end{proposition}
\begin{proof}
    By Observation~\ref{obs:closed-neigh-hg-predomination}, the winning sets for Staller in the MBD game on $(G,D)$ are the closed neighborhoods of the vertices from $V(G) \setminus D$. Thus, removing an edge $e=uu'$ with $u,u' \in D$ does not modify the winning sets of the game. It implies that the game's outcome is the same on $(G,D)$ and $(G-e, D)$. Further, for every vertex $v \in V(G) \setminus D$, the same equivalence holds for $(G, D \cup \{v\})$ and $(G-e, D \cup \{v\})$. It finishes the proof for part (i).

    To prove (ii), we remark that a predominated and isolated vertex $u$ does not belong to any winning sets and therefore, it is an isolated vertex in $\cN(G,D)$. Then, Proposition~\ref{prop:critical-hgs-properties}~(ii) directly implies the statement. 

    To show~(iii), we remark that those components of $G$ that are entirely predominated correspond to a set of isolated vertices in the hypergraph $\cN(G,D)$. Then, by Proposition~\ref{prop:critical-hgs-properties}~(iii), part (iii) follows.
\qed    
 \end{proof}

\begin{definition} \label{def:atomic-MBD}
  A predominated graph $(G,D)$ is an \emph{atomic MBD critical} graph if $(G,D)$ is MBD critical, $D$ is an independent set, and no vertex from $D$ is an isolated vertex in $G$. 
\end{definition}

Clearly, $(P_1, \emptyset)$ is an atomic MBD critical graph. By Proposition~\ref{propositions}~(iii) and Definition~\ref{def:atomic-MBD}, it is the only atomic MBD critical graph that contains an isolated vertex. If $(G,D)$ is MBD critical, we may obtain an atomic MBD critical graph $(G', D')$ by removing all the edges inside $D$ and then deleting all the isolated vertices. If $G$ is a disconnected graph, then by Proposition~\ref{propositions}~(iii), the obtained atomic MBD critical graph $(G', D')$ will be a connected subgraph of one component of $G$.

%%%%%%%%%%%%%%%%%%%%%%%%%
\section{MBD critical trees}
\label{sec:MBD-critical-trees}

%%%%%%%%%
\subsection{Preliminaries}

In~\cite{bujtas-2023}, the authors considered the minimum number $ \gsmb'(G) $ of moves Staller needs to win the game on a graph $G$ (provided that both players play optimally). The characterization of trees $T$ having $\gsmb'(T)=k$ involved a definition of a family  $\cS_k$ of graphs where every graph $S \in \cS_k$ is given together with a fixed subset of vertices denoted by $X(S)$. Here, without introducing $\cS_k$ for each $k$, we consider the union $\cS$ of these families. According to \cite[Proposition 3.4]{bujtas-2023}, this union may also be defined as follows.

\begin{definition} \label{def:cS}
   For a tree $T$, let $S(T)$ denote the tree obtained from $T$ by subdividing each edge exactly once. We define
  $$ \cS=\{ S(T): \mbox{$T$ is a tree} \} \quad \mbox{and} \quad X(S(T))=V(T).
  $$     
\end{definition}

We remark that $S(P_1)=P_1$ and $X(P_1) = V(P_1)$.

\begin{definition} \label{def:substructure}
  We say that $F \in \cS$ is a \emph{substructure} in the graph $G$, if $F$ is a subgraph of $G$ and $\deg_G(v) = \deg_F(v)$ holds for each $v \in X(F)$.
\end{definition}

Motivated by the requirement in Definition~\ref{def:substructure}, the vertices in $X(F)$ are called \emph{fixed-degree vertices}. Given a substructure $F$ in a graph $G$, we will say that a vertex $v \in V(F)$ is \emph{black} if $v$ is a fixed-degree vertex (i.e., $v \in X(F)$), and otherwise (i.e., if $v \in V(F) \setminus X(F)$) vertex $v$ is \emph{white}. In Fig.~\ref{fig:T-S(T)-subsctructure} these definitions are illustrated. 

\begin{figure}[ht!]
\begin{center}
\begin{tikzpicture}[scale=0.7,style=thick,x=1cm,y=1cm]
\def\vr{5pt}
\begin{scope}[xshift=-1cm, yshift=0cm] % Graph T
% vertices defined
\coordinate(v1) at (0.0,2.0);
\coordinate(v2) at (-2.0,0.0);
\coordinate(v3) at (0.0,0.0);
\coordinate(v4) at (2.0,0.0);
\coordinate(v5) at (0.0,-2.0);
\coordinate(v6) at (0.0,-4.0);
% \edges		
\draw (v2) -- (v3) -- (v4);  
\draw (v1) -- (v3) -- (v5) -- (v6);  

\foreach \i in {1,2,3,4,5,6} 
{
\draw(v\i)[fill=white] circle(\vr);
}
% text
\node at (-2.0,2.0) {$T$};

\end{scope}
\begin{scope}[xshift=6cm, yshift=0cm] % Graph S(T)
% vertices defined
\coordinate(v1) at (0.0,2.0);
\coordinate(v2) at (-2.0,0.0);
\coordinate(v3) at (0.0,0.0);
\coordinate(v4) at (2.0,0.0);
\coordinate(v5) at (0.0,-2.0);
\coordinate(v6) at (0.0,-4.0);
\coordinate(v7) at (0.0,1.0);
\coordinate(v8) at (-1.0,0.0);
\coordinate(v9) at (1.0,0.0);
\coordinate(v10) at (0.0,-3.0);
\coordinate(v11) at (0.0,-1.0);
% \edges		
\draw (v2) -- (v3) -- (v4);  
\draw (v1) -- (v3) -- (v5) -- (v6);  

%  vertices
\foreach \i in {1,2,3,4,5,6} 
{
\draw(v\i)[fill=black] circle(\vr);
}
\foreach \i in {7,8,9,10,11} 
{
\draw(v\i)[fill=white] circle(\vr);
}

% text
\node at (-2.0,2.0) {$S(T)$};
\end{scope}
\begin{scope}[xshift=13cm, yshift=0cm] % Substructure S(T)
% vertices defined
\coordinate(v1) at (0.0,2.0);
\coordinate(v2) at (-2.0,0.0);
\coordinate(v3) at (0.0,0.0);
\coordinate(v4) at (2.0,0.0);
\coordinate(v5) at (0.0,-2.0);
\coordinate(v6) at (0.0,-4.0);
\coordinate(v7) at (0.0,1.0);
\coordinate(v8) at (-1.0,0.0);
\coordinate(v9) at (1.0,0.0);
\coordinate(v10) at (0.0,-3.0);
\coordinate(v11) at (0.0,-1.0);
\coordinate(v12) at (-1.0,-1.0);
\coordinate(v13) at (-2.0,-1.0);
\coordinate(v14) at (-3.0,-1.0);
\coordinate(v15) at (-2.0,-2.0);
\coordinate(v16) at (-2.5,-3.0);
\coordinate(v17) at (-1.5,-3.0);
\coordinate(v18) at (1.0,-1.0);
\coordinate(v19) at (2.0,-1.0);
\coordinate(v20) at (1.0,-2.0);

% \edges		
\draw (v2) -- (v3) -- (v4);  
\draw (v1) -- (v3) -- (v5) -- (v6); 
\draw (v12) -- (v8);
\draw (v12) -- (v13) -- (v14);  
\draw (v13) -- (v15); 
\draw (v15) -- (v16); 
\draw (v15) -- (v17); 
\draw (v11) -- (v18) -- (v19); 
\draw (v18) -- (v20); 
\draw[rounded corners,dashed](-2.5,-0.5)--(-2.5,0.5)--(-0.5,0.5)--(-0.5,2.5)--(0.5,2.5)--(0.5,0.5)--(2.5,0.5)--(2.5,-0.5)--(0.5,-0.5)--(0.5,-4.5)--(-0.5,-4.5)--(-0.5,-0.5)--cycle;

%  vertices
\foreach \i in {1,2,3,4,5,6} 
{
\draw(v\i)[fill=black] circle(\vr);
}
\foreach \i in {7,8,9,10,11} 
{
\draw(v\i)[fill=white] circle(\vr);
}
\foreach \i in {12,13,14,15,16,17,18,19,20} 
{
\draw(v\i)[fill=white] circle(\vr);
}

% text
\node at (-2.0,2.0) {$G$};
\end{scope}
\end{tikzpicture}
\caption{A tree $T$, the tree $S(T)$, and $G$ with a substructure $S(T)$.
}
\label{fig:T-S(T)-subsctructure}
\end{center}
\end{figure}
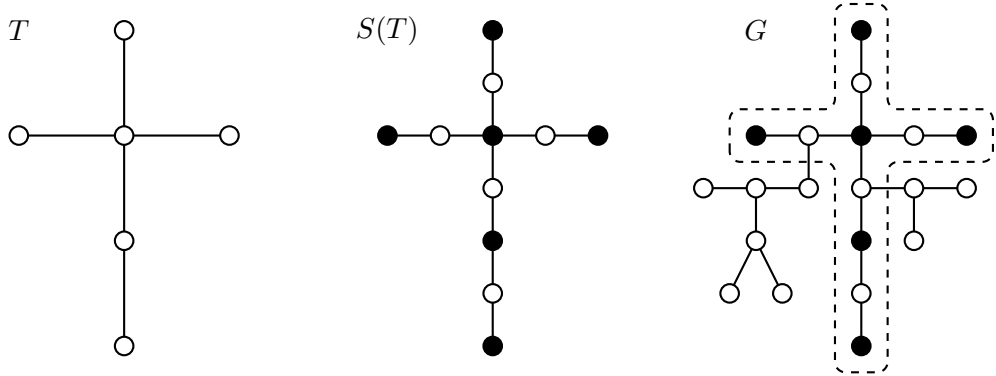

With this terminology, we may state the following result which can be quickly derived from~\cite[Theorem 3.5]{bujtas-2023}.

\begin{theorem}[\cite{bujtas-2023}] \label{thm:cS}
   Staller wins the MBD game on a tree $T$, if and only if\/ $T$ contains a substructure $F \in \cS$.  
\end{theorem}

Furthermore, it turns out that we can utilize~\cite[Proposition~3.7]{bujtas-2023} to establish a sufficient condition for a hypergraph $\cH$ on which Maker can win the Maker-Breaker game. Using our notation introduced for closed neighborhood hypergraphs of predominated graphs, for every $F \in \cS$, the hypergraph $\cN(F, V(F) \setminus X(F))$ consists of hyperedges which are the closed neighborhoods of the vertices in $X(F)$. Then, we may define 
$$ \cF= \{\cN(F, V(F) \setminus X(F)): F \in \cS    \}
$$
and restate the result from~\cite{bujtas-2023}.

\begin{proposition}[\cite{bujtas-2023}] \label{prop:cF}
   If a hypergraph $\cH$ contains a subhypergraph from $\cF$, then Maker can win the Maker-Breaker game on $\cH$.
\end{proposition}

%%%%%%%%%%%%
\subsection{Characterization of MBD critical trees}

\begin{proposition}   \label{prop:coloring-substructures}
    Suppose that $T$ is a tree, $F_1 \in \cS$ and $F_2 \in \cS$ are substructures in $T$, and $v$ is a common vertex of $F_1$ and $F_2$. Then, $v \in X(F_1)$ if and only if $v \in X(F_2)$. 
\end{proposition}
\begin{proof}
The statement clearly holds if $F_1$ is an isolated vertex $P_1$ as, in this case, $T$ and $F_2$ are also isolated vertices. From now on we may suppose that both $F_1$ and $F_2$ are of order at least $3$.
    By Definitions~\ref{def:cS} and~\ref{def:substructure}, every nontrivial substructure $F \in \cS$ in $T$ contains at least two leaves from $T$, and all the leaves in $F$ belong to $X(F)$. The definitions also imply that the graph $F$ is bipartite, and that the bipartition of its vertices is given by $X(F)$ and $V(F) \setminus X(F)$, that is, the set of black and white vertices in $F$.

    Consider now the two substructures $F_1$ and $F_2$ in $T$ that share some vertices. Since $T$ is a tree and $F_1$ and $F_2$ are connected, the common vertices of $F_1$ and $F_2$ induce a connected subgraph $F_{1,2}$ of $T$. If $F_{1,2}$ is a tree of order at least two, we can choose a leaf $u$ in it. If $F_{1,2}$ consists of only one vertex, choose this vertex as $u$.
    
    If $u$ is a leaf in $F_1$, it is a fixed-degree vertex in the substructure, and therefore, it is a leaf in $T$ and $F_2$ as well. It follows that $u \in X(F_1) \cap X(F_2)$. Then, the bipartition of the substructures gives that for every $v \in V(F_1) \cap V(F_2)$, vertex $v$ belongs to $X(F_1)$ if and only if $v \in X(F_2)$. The same is true if we start with the assumption that $u$ is a leaf in $F_2$.

    If $u$ is a leaf or the only vertex in $F_{1,2}$ but it is neither a leaf in $F_1$ nor in $F_2$, then $u$ has a neighbor from both $V(F_1) \setminus V(F_2)$ and $V(F_2) \setminus V(F_1)$. Thus, $u$ is neither a fixed-degree vertex in $F_1$ nor in $F_2$. We then have  $u \in V(F_i) \setminus X(F_i)$, for $i \in \{1,2\}$, and the bipartitions of $F_1$ and $F_2$ give that, for every $v \in V(F_{1,2})$, the color of $v$ in $F_1$ is the same as its color in $F_2$.  \qed    
\end{proof}

We now define the following vertex coloring for a tree $T$. A vertex $v$ is \emph{black} in $T$ if there is a substructure $F \in \cS$ in $T$ such that $v$ is black in that substructure, i.e.~$v \in X(F)$; a vertex $v$ is \emph{white} if there exists a substructure $F$ in $T$ such that $v$ is white in that substructure, i.e.~$v \in V(F) \setminus X(F)$. The remaining vertices of $T$ belong to no substructure, and we say that they are \emph{gray}. 

By Proposition~\ref{prop:coloring-substructures} the vertex colorings stemming from different substructures are compatible -- no vertex from $T$ is colored with both white and black. Therefore, the sets of black, white, and gray vertices give a partition of $V(T)$. An example of a tree $T$ with its substructures and its vertex coloring is presented in Fig.~\ref{fig:substructures}.
%The tree $T$ contains $4$ substructures, they are also displayed. 
%We remark that this assignment colors every substructure $F$ with black and white according to $X(F)$ and $V(F) \setminus X(F)$, as it was introduced earlier. 

\begin{figure}[ht!]
\begin{center}
\begin{tikzpicture}[scale=0.7,style=thick,x=1cm,y=1cm]
\def\vr{5pt}
% vertices defined
\coordinate(v1) at (2.0,0.0);
\coordinate(v2) at (4.0,0.0);
\coordinate(v4) at (2.0,2.0);
\coordinate(v5) at (4.0,2.0);
\coordinate(v6) at (6.0,2.0);
\coordinate(v7) at (2.0,4.0);
\coordinate(v8) at (4.0,4.0);
\coordinate(v9) at (6.0,4.0);
\coordinate(v10) at (4.0,6.0);
\coordinate(v11) at (0.5,7.5);
\coordinate(v12) at (2.5,7.5);
\coordinate(v13) at (5.5,7.5);
\coordinate(v14) at (7.5,7.5);
% \edges		
\draw (v2) -- (v5) -- (v4);  
\draw (v6) -- (v5) -- (v8) -- (v7);  
\draw (v9) -- (v8) -- (v10) -- (v12) -- (v11);  
\draw (v10) -- (v13) -- (v14);  
\draw (v10) -- (v12) -- (v11);  
\draw (v1) -- (v4);
\draw[rounded corners, densely dotted](-0.2,3.2)--(-0.2,8.4)--(8.0,8.4)--(8.0,7.2)--(5.5,7.2)--(4.5,6.2)--(4.5,3.2)--cycle;
\draw[rounded corners,dashed](8.4,3.0)--(8.4,8.0)--(0.2,8.0)--(0.2,7.0)--(2.7,7.0)--(3.7,6.0)--(3.7,3.0)--cycle;
\draw[rounded corners](1.5,3.4)--(1.5,4.6)--(7.5,4.6)--(7.5,3.4)--cycle;
\draw[rounded corners](3.5,-0.5)--(3.5,2.5)--(7.5,2.5)--(7.5,-0.5)--cycle;
% \draw (v12) .. controls (-4,3.3) and (-2,3.3) .. (v14);
% \draw plot [smooth cycle] coordinates {(-5.0,3.0) (-5.3,4.5) (-4.7,4.5)};
%  vertices
\foreach \i in {2,6,7,9,10,11,14} 
{
\draw(v\i)[fill=black] circle(\vr);
}
\foreach \i in {5,8,12,13} 
{
\draw(v\i)[fill=white] circle(\vr);
}
\foreach \i in {1,4} 
{
\draw(v\i)[fill=lightgray] circle(\vr);
}
% \draw(v\i)[fill=lightgray] circle(\vr);
%\path [fill=white] (0.1,0.1) rectangle (0.5,0.5);
% \filldraw[fill=white, draw=black, line width=0.3mm] (1.8, -0.2) rectangle +(0.4, 0.4);
% \filldraw[fill=white, draw=black, line width=0.3mm] (1.8, 1.8) rectangle +(0.4, 0.4);
% text
%\node at (1.0,6.0) {$F_1$};
%\node at (7.0,6.0) {$F_2$};
%\node at (7.0,4.0) {$F_3$};
%\node at (6.5,0.5) {$F_4$};
%
\end{tikzpicture}
\caption{The partition of the vertices of the tree $T$ into black, white and gray vertices, along with the four substructures of $T$. The colorings on substructures are compatible on their intersections, by Proposition~\ref{prop:coloring-substructures}.
}
\label{fig:substructures}
\end{center}
\end{figure}
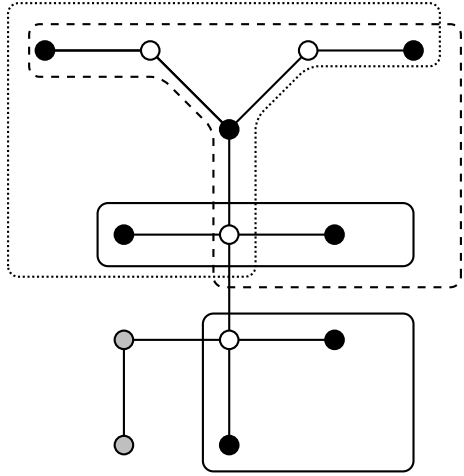

\begin{theorem} \label{thm:tree-St-wins}
  Staller wins on a predominated tree $(T,D)$ if and only if $T$ contains a substructure $F \in \cS$ such that no vertex from $X(F)$ belongs to $D$.  
\end{theorem}
\begin{proof}
    Suppose first that a predominated tree $(T,D)$ contains a substructure $F \in \cS$ such that $X(F) \cap D = \emptyset$. Since every vertex $v \in X(F)$ has the same degree in $T$ and $F$, the closed neighborhoods $N_F[v]$ and $N_T[v]$ are the same. As $v \notin D$ for every $v \in X(F)$, the closed neighborhoods of the black vertices of $F$ are all present in $\cN(T, D)$. It implies that $\cN(T, D)$ contains the subhypergraph $\cN(F, V(F)\setminus X(F))$ and, by Proposition~\ref{prop:cF}, we may conclude that Staller wins the game.

    In the second part of the proof, we suppose that $D$ contains at least one black vertex from each substructure in $T$ and show that Dominator wins the game. We proceed by induction on $k = n(T)+ |D|$. Under the present condition, the smallest value of $k$ is $2$. It covers two cases. First, if $T$ is just an isolated vertex, it is a substructure with one black vertex, which is predominated. Then, Dominator wins the game, as there is no winning set for Staller. The second case is $(P_2, \emptyset)$ when the condition is satisfied as $P_2$ contains no substructure from $\cS$. Clearly, Dominator wins the MBD game on $(P_2, \emptyset)$.

    Assume that $k \ge 3$ and the statement holds for every predominated tree $(T', D')$ with $n(T') +|D'| \leq k-1$. Consider a predominated tree $(T,D)$ where $n(T) +|D|=k$ and every substructure $F \in \cS$ in $T$ contains a predominated black vertex. If every black leaf of $T$ is predominated, then we remove a black leaf $x$ from $D$. Since every substructure contains at least two leaves, it remains true in $(T, D\setminus \{ x\})$ that every substructure has a leaf in $D$. Hence, we can apply the hypothesis and get that Dominator wins on $(T, D\setminus \{ x\})$. Therefore, Dominator also wins on $(T,D)$ where more vertices are predominated.

    The other case is when a black leaf $x$ in $T$ is not predominated. Let $y$ be the neighbor of $x$. Consider the components $T_0, \dots, T_{\ell}$ of $T-y$ such that $T_0$ is the isolated vertex $x$. Let $x_0=x, \dots , x_\ell$, respectively, be the vertices in the components $T_0, \dots, T_\ell$ that are adjacent to $y$ in $T$. We continue to refer to the colors of vertices as they were defined in $T$. In particular, $x$ is black and $y$ is white. We prove the following claim to show that the hypothesis can be applied to $T-y$.

    \begin{claim} \label{claim:1}
        Every substructure $F \in \cS$ in $T-\{x,y\}$ contains a vertex in $X(F) \cap D$.
    \end{claim}
   \proof Let $F$ be a substructure in $T_i$, for $i \in [\ell]$. If $F$ does not contain $x_i$, then $F$ is also a substructure in $T$ and, by our condition, $X(F) \cap D$ is not empty. If $x_i$ is not a fixed-degree vertex in $F$, then again, $F$ is a substructure in $T$ and satisfies the condition $X(F) \cap D \neq \emptyset$. If $x_i$ is a fixed-degree vertex in $F$, then consider the tree $F^+$ obtained by adding vertices $y$ and $x$ to $F$; that is $F^+$ is the subtree induced by $V(F) \cup \{x,y\}$ in $T$. As $F \in \cS$, it is a subdivision of a tree with $x_i$ being a non-subdivision vertex. Therefore, $F^+$ can also be obtained as a subdivision of a tree, and $X(F^+)= X(F) \cup \{x\}$. We conclude that $\cF^+$ is a substructure in $T$ and, since $x \notin D$, a fixed-degree vertex of $F$ belongs to $D$. \smallqed

   By Claim~\ref{claim:1}, the hypothesis can be applied to each component, and we infer that Dominator can win on $T_1, \dots , T_\ell$ (even if Staller starts the game). Clearly, Dominator also wins on the one-edge graph induced by $x$ and $y$. It implies that Dominator wins on the union of these trees. Finally, putting back the edges $yx_i$, for $i \in [\ell]$, we obtain $T$ where Dominator can also win. It finishes the proof of the theorem. \qed
   \end{proof}

\begin{theorem} \label{thm:critical-trees}
If $(T, D)$ is a predominated tree, then the following holds. 
\begin{itemize}
    \item[$(i)$]  $(T, D)$ is  MBD critical if and only if there exists a substructure $F \in \cS$ in $T$ such that $D= V(T) \setminus X(F)$.
    \item[$(ii)$] $(T, D)$ is an atomic MBD critical tree if and only if $T \in \cS$ and $D= V(T) \setminus X(T)$.
\end{itemize}
  \end{theorem}
\begin{proof}
    (i) Suppose first that $F \in \cS$ is a substructure in $T$ and $D= V(T) \setminus X(F)$. As no black vertex from $F$ belongs to $D$, Theorem~\ref{thm:tree-St-wins} implies that Staller wins the game on $(T,D)$. Further, the only possibility of extending $D$ is adding vertices from $X(F)$ to it.  For every $v \in X(F)$, the set $D'= D \cup \{v\}$ contains a black vertex from every substructure of $T$. To see this, observe that, by definition of a substructure, no proper subset of $V(F)$ induces a substructure in $T$. Applying Theorem~\ref{thm:tree-St-wins} again, we conclude that Dominator wins the game on $(T, D \cup \{v\})$ for every $v \in V(T) \setminus D $, and hence $(T,D)$ is MBD critical.

    Suppose now that $(T,D)$ is an MBD critical predominated tree. By Theorem~\ref{thm:tree-St-wins}, there is at least one substructure $F \in \cS$ in $T$ such that no black vertex of $F$ belongs to $D$. Assume for a contradiction that some further vertices of $T$ are also omitted from $D$. Let $v$ be such a vertex. That is, $v \notin D$ and $v \notin X(F)$. Now,  Theorem~\ref{thm:tree-St-wins} implies that Staller can win the game on $(T, D \cup \{v\})$ as the black vertices of $F$ remain outside $D \cup \{v\}$. It contradicts the criticality of $(T,D)$ and, in turn, finishes the proof of~(i).

    (ii)  An atomic MBD critical tree $(T, D)$ surely satisfies the necessary and sufficient condition for MBD critical trees that we just proved in~(i), ensuring the existence of a substructure $F \in \cS$ in $T$ such that $D= V(T) \setminus X(F)$. Furthermore, by definition of an atomic MBD critical tree there is no edge inside $D$. Thus, every edge of $T$ is incident to at least one black vertex from $F$. As these are the fixed-degree vertices of $F$, no edge of $T$ is incident to vertices outside $F$. Since $(T,D)$ is an atomic MBD critical tree, it contains no isolated vertex. Then $T$ itself is the substructure $F$ and, by part~(i),  $D= V(T) \setminus X(T)$ as stated.

    To prove the other direction of~(ii), we observe that (i) implies the MBD criticality of $(T,D)$, as $T \in \cS$ and $D$ is the set of white vertices of $T$. It is also clear that no substructure different from $P_1$ contains isolated vertices or edges between two white vertices. It follows that, under the given conditions, $(T,D)$ is an atomic MBD critical graph. This finishes the proof of (ii). \qed
\end{proof}

Theorem~\ref{thm:critical-trees} is illustrated in Fig.~\ref{fig:atomic-critical-2x}, where two predominated trees satisfying (i), and one predominated tree satisfying (ii) are shown. In the figure, predominated vertices are marked with a square around them. 

\begin{figure}[ht!]
\begin{center}
\begin{tikzpicture}[scale=0.7,style=thick,x=1cm,y=1cm]
\def\vr{5pt}
\begin{scope}[xshift=0cm, yshift=0cm] % 1
 % atomic 
% vertices defined
\coordinate(v1) at (0.0,0.0);
\coordinate(v2) at (2.0,0.0);
\coordinate(v3) at (4.0,0.0);
\coordinate(v4) at (6.0,0.0);
\coordinate(v5) at (8.0,0.0);
\coordinate(v6) at (10.0,0.0);
\coordinate(v7) at (12.0,0.0);
\coordinate(v8) at (8.0,2.0);
\coordinate(v9) at (10.0,2.0);
% \edges		
\draw (v1) -- (v7);  
\draw (v5) -- (v8) -- (v9);  
%  vertices
\filldraw[fill=white, draw=black, line width=0.3mm] (1.7, -0.3) rectangle +(0.6, 0.6);
\filldraw[fill=white, draw=black, line width=0.3mm] (5.7, -0.3) rectangle +(0.6, 0.6);
\filldraw[fill=white, draw=black, line width=0.3mm] (9.7, -0.3) rectangle +(0.6, 0.6);
\filldraw[fill=white, draw=black, line width=0.3mm] (7.7, 1.7) rectangle +(0.6, 0.6);

\foreach \i in {1,3,5,7,9} 
{
\draw(v\i)[fill=black] circle(\vr);
}
\foreach \i in {2,4,6,8} 
{
\draw(v\i)[fill=white] circle(\vr);
}
\end{scope}
\begin{scope}[xshift=0cm, yshift=15cm] % first critical 
% vertices defined
\coordinate(v1) at (0.0,0.0);
\coordinate(v2) at (2.0,0.0);
\coordinate(v3) at (4.0,0.0);
\coordinate(v4) at (6.0,0.0);
\coordinate(v5) at (8.0,0.0);
\coordinate(v6) at (10.0,0.0);
\coordinate(v7) at (12.0,0.0);
\coordinate(v8) at (8.0,2.0);
\coordinate(v9) at (10.0,2.0);
\coordinate(v10) at (2.0,-2.0);
\coordinate(v11) at (0.5,-3.5);
\coordinate(v12) at (2.0,-3.5);
\coordinate(v13) at (3.5,-3.5);
% \edges		
\draw (v1) -- (v7);  
\draw (v5) -- (v8) -- (v9);  
\draw (4,2) -- (8,2); 
\draw (v11) -- (v10) -- (v2);
\draw (v12) -- (v10) -- (v13);
\draw[rounded corners](-0.7,-0.7)--(-0.7,0.7)--(7.3,0.7)--(7.3,2.7)--(12.7,2.7)--(12.7,-0.7)--cycle;
%  vertices
% FIRST DUUBLE SQUARE
\filldraw[fill=white, draw=black, line width=0.3mm] (5.7, 1.7) rectangle +(0.6, 0.6);
\draw(6.00, 2.0)[fill=lightgray] circle(\vr);

% Second DUUBLE SQUARE
\filldraw[fill=white, draw=black, line width=0.3mm] (3.7, 1.7) rectangle +(0.6, 0.6);
\draw(4.00, 2.0)[fill=lightgray] circle(\vr);

% other squares
\filldraw[fill=white, draw=black, line width=0.3mm] (1.7, -0.3) rectangle +(0.6, 0.6);
\filldraw[fill=white, draw=black, line width=0.3mm] (1.7, -2.3) rectangle +(0.6, 0.6);
\filldraw[fill=white, draw=black, line width=0.3mm] (5.7, -0.3) rectangle +(0.6, 0.6);
\filldraw[fill=white, draw=black, line width=0.3mm] (9.7, -0.3) rectangle +(0.6, 0.6);
\filldraw[fill=white, draw=black, line width=0.3mm] (7.7, 1.7) rectangle +(0.6, 0.6);
\filldraw[fill=white, draw=black, line width=0.3mm] (0.2, -3.8) rectangle +(0.6, 0.6);
\filldraw[fill=white, draw=black, line width=0.3mm] (1.7, -3.8) rectangle +(0.6, 0.6);
\filldraw[fill=white, draw=black, line width=0.3mm] (3.2, -3.8) rectangle +(0.6, 0.6);
\foreach \i in {1,3,5,7,9,11,12,13} 
{
\draw(v\i)[fill=black] circle(\vr);
}
\foreach \i in {2,4,6,8,10} 
{
\draw(v\i)[fill=white] circle(\vr);
}
\node at (12,1.5) {$F$};
\end{scope}
\begin{scope}[xshift=0cm, yshift=7cm] % second critical 
% vertices defined
\coordinate(v1) at (0.0,0.0);
\coordinate(v2) at (2.0,0.0);
\coordinate(v3) at (4.0,0.0);
\coordinate(v4) at (6.0,0.0);
\coordinate(v5) at (8.0,0.0);
\coordinate(v6) at (10.0,0.0);
\coordinate(v7) at (12.0,0.0);
\coordinate(v8) at (8.0,2.0);
\coordinate(v9) at (10.0,2.0);
\coordinate(v10) at (2.0,-2.0);
\coordinate(v11) at (0.5,-3.5);
\coordinate(v12) at (2.0,-3.5);
\coordinate(v13) at (3.5,-3.5);
% \edges		
\draw (v1) -- (v7);  
\draw (v5) -- (v8) -- (v9);  
\draw (4,2) -- (8,2); 
\draw (v11) -- (v10) -- (v2);
\draw (v12) -- (v10) -- (v13);
\draw[rounded corners](1.3,-4.0)--(1.3,-1.4)--(4.0,-1.4)--(4.0,-4.0)--cycle;
%  vertices
% FIRST DUUBLE SQUARE
\filldraw[fill=white, draw=black, line width=0.3mm] (5.7, 1.7) rectangle +(0.6, 0.6);
\draw(6.00, 2.0)[fill=lightgray] circle(\vr);

% Second DUUBLE SQUARE
\filldraw[fill=white, draw=black, line width=0.3mm] (3.7, 1.7) rectangle +(0.6, 0.6);
\draw(4.00, 2.0)[fill=lightgray] circle(\vr);
% other squares
\filldraw[fill=white, draw=black, line width=0.3mm] (1.7, -0.3) rectangle +(0.6, 0.6);
\filldraw[fill=white, draw=black, line width=0.3mm] (1.7, -2.3) rectangle +(0.6, 0.6);
\filldraw[fill=white, draw=black, line width=0.3mm] (5.7, -0.3) rectangle +(0.6, 0.6);
\filldraw[fill=white, draw=black, line width=0.3mm] (9.7, -0.3) rectangle +(0.6, 0.6);
\filldraw[fill=white, draw=black, line width=0.3mm] (7.7, 1.7) rectangle +(0.6, 0.6);
\filldraw[fill=white, draw=black, line width=0.3mm] (0.2, -3.8) rectangle +(0.6, 0.6);
%\filldraw[fill=white, draw=black, line width=0.3mm] (1.7, -3.8) rectangle +(0.6, 0.6);
%\filldraw[fill=white, draw=black, line width=0.3mm] (3.2, -3.8) rectangle +(0.6, 0.6);
\filldraw[fill=white, draw=black, line width=0.3mm] (9.7, 1.7) rectangle +(0.6, 0.6);
\filldraw[fill=white, draw=black, line width=0.3mm] (-0.3, -0.3) rectangle +(0.6, 0.6);
\filldraw[fill=white, draw=black, line width=0.3mm] (3.7, -0.3) rectangle +(0.6, 0.6);
\filldraw[fill=white, draw=black, line width=0.3mm] (7.7, -0.3) rectangle +(0.6, 0.6);
\filldraw[fill=white, draw=black, line width=0.3mm] (11.7, -0.3) rectangle +(0.6, 0.6);

\foreach \i in {1,3,5,7,9,11,12,13} 
{
\draw(v\i)[fill=black] circle(\vr);
}
\foreach \i in {2,4,6,8,10} 
{
\draw(v\i)[fill=white] circle(\vr);
\node at (3.5,-2) {$F'$};
}\end{scope}

\end{tikzpicture}
\caption{Two MBD critical predominated trees (above) and an atomic MBD critical predominated tree. The substructures $F$ and $F'$ as in Theorem~\ref{thm:critical-trees}~(i) are also marked. The predominated vertices are marked with a square around them.}
\label{fig:atomic-critical-2x}
\end{center}
\end{figure}
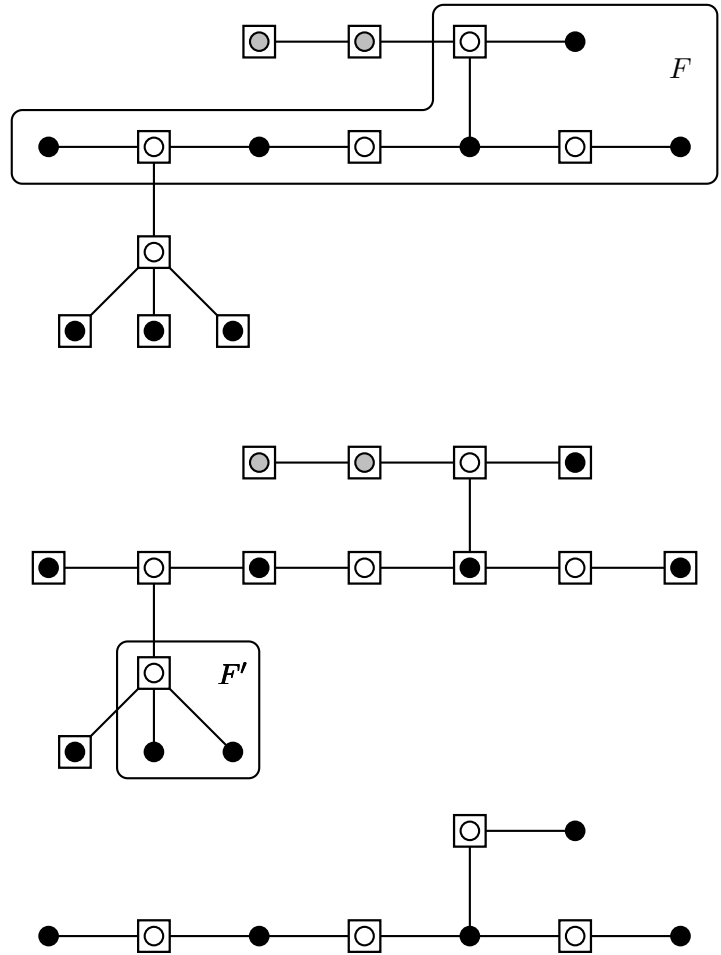

\begin{remark}
  It is straightforward to convince oneself that Theorems~\ref{thm:tree-St-wins} and~\ref{thm:critical-trees} remain valid over the class of predominated forests. 
\end{remark}

From the algorithmic point of view, Theorem~\ref{thm:critical-trees} leads to the the following result. 

\begin{theorem}
\label{thm:critical-tree-linear}
It can be checked in linear time whether a given predominated tree is MBD critical. 
\end{theorem}

\begin{proof}
Let us observe the procedure given in Algorithm~\ref{alg:trees}. 

\medskip
\begin{algorithm}[H]
    \KwIn{Predominated tree $(T,D)$}
    \KwOut{True if $(T,D)$ is MBD critical, false otherwise} \vspace{2mm}
    $X$ = $V(T)\setminus D$\;
    $T'$ = $T[\cup_{x\in X}N[x]]$\;
  \uIf{$T'$ is connected}{
    let $(V_0',V_1')$ be the bipartition of $T'$\;
    \uIf{$X=V_i'$, for some $i\in \{0,1\}$, \textbf{and} $\deg_{T'}(y) = 2$, for every $y\in V_{1-i}'$, }{
    \KwRet{true}\;
  }
  \uElse{
    \KwRet{false}\;
  }}
  \uElse{
    \KwRet{false}\;
  }
  
   \caption{Fast recognition of predominated MBD critical trees.}
\label{alg:trees}
\end{algorithm}

\medskip
By Theorem~\ref{thm:critical-trees}, $(T, D)$ is  MBD critical if and only if there exists a substructure $F \in \cS$ in $T$ such that $D= V(T) \setminus X(F)$. If so, then $X = V(T)\setminus D$ must form the fixed-degree vertices of the substructure $F$. Since $F$ is connected, the subtree $T'= T[\cup_{x\in X}N[x]]$ must also be connected. If $(V_0',V_1')$ is the bipartition of $T'$, then all the vertices of $X$ form one bipartition set, say $V_i'$, as otherwise two vertices from $X$ would be at an odd distance, but this cannot happen in a substructure $F \in \cS$. 

Moreover, the condition $\deg_{T'}(y) = 2$ for every $y\in V_{1-i}'$ implies that each vertex of $V(F)\setminus X$ is a subdivision vertex. It follows that $F \in \cS$ and since $D= V(T) \setminus X(F)$, Theorem~\ref{thm:critical-trees} implies that $(T,D)$ is MBD critical if and only if the conditions checked by the algorithm are fulfilled. 

As for the time complexity, it is straightforward to see that all the tasks in Algorithm~\ref{alg:trees} can be performed in linear time. 
\qed
\end{proof}

%%%%%%%%%%%%%%%%%%%%%%%%%%%%%%
\section{MBD critical cactus graphs}
\label{sec:MBD-critical-cactuses}

\begin{definition} \label{def:cC}
   The \emph{double-odd replacement} of an edge $e=xy$ in a graph $G$ means removing the edge $e$ and replacing it with two internally vertex-disjoint $x,y$-paths both of odd length.
   
   For an $F \in \cS$, we say that the cactus graph $H$ is an \emph{$F$-cactus} if either $H \cong F$, or $H$ can be obtained from $F$ by double-odd 
   replacements of some edges of $F$. 
   
   The set $\cC$ of \emph{cactus-substructures} is then defined as follows:
   $$ \cC= \{ H: H \mbox{ is an $F$-cactus for some } F \in \cS \}.$$
   
   For an $F$-cactus $H$, the set of \emph{fixed-degree vertices} is specified as $X(H)$ corresponding to the bipartition class of $H$ that contains $X(F)$. 

   We say that a cactus $H \in \cC$ is a {\em substructure} in a graph $G$ if $H$ is a subgraph of $G$ and  $\deg_G(v)=\deg_H(v)$ holds for every $v \in X(H)$. 
\end{definition}

We emphasize that in the above definition, a substructure is not necessarily an induced subgraph. 

Note that a double-odd replacement of an edge $e=xy$ in $F \in \cS$ replaces $e$ with an even cycle $C_e$. It might happen that $C_e$ is a $2$-cycle; this case is equivalent to having only one edge between $x$ and $y$. As $F$ is a bipartite graph, vertices $x$ and $y$ belong to different classes in $F$. The double-odd replacement of $e$ keeps that property. 

In particular, we have that an $F$-cactus is bipartite, for every $F \in \cS$. This further implies that \emph{fixed-degree vertices} are well defined, as all the vertices in $X(F)$ remain in the same bipartition class after every double-odd replacement.

Let us say that the vertices in $X(H)$ are {\em black} while the remaining vertices of $H$ are {\em white}. The color classes then correspond to the bipartition of $H$, see Fig.~\ref{fig:cactus}. 

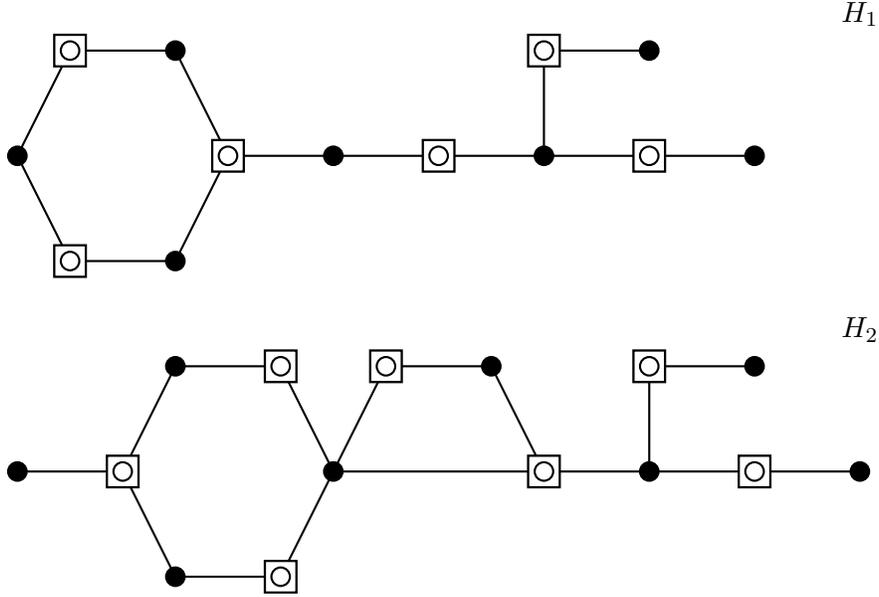
\begin{figure}[ht!]
\vspace*{1cm}
\begin{center}
\begin{tikzpicture}[scale=0.7,style=thick,x=1cm,y=1cm]
\def\vr{5pt}

\begin{scope}[xshift=0cm, yshift=0cm] % first critical F-cactus H1 
% vertices defined
\coordinate(v1) at (-2.0,0.0);
\coordinate(v2) at (2.0,0.0);
\coordinate(v3) at (4.0,0.0);
\coordinate(v4) at (6.0,0.0);
\coordinate(v5) at (8.0,0.0);
\coordinate(v6) at (10.0,0.0);
\coordinate(v7) at (12.0,0.0);
\coordinate(v8) at (8.0,2.0);
\coordinate(v9) at (10.0,2.0);
\coordinate(v10) at (-1, 2.0);
\coordinate(v11) at (1.0, 2.0);
\coordinate(v12) at (-1.0,-2.0);
\coordinate(v13) at (1.0,-2.0);
% \edges		
\draw (v2) -- (v7);  
\draw (v5) -- (v8) -- (v9);  
\draw (v1) -- (v10) -- (v11) -- (v2);
\draw (v1) -- (v12) -- (v13) -- (v2);
%  vertices

% squares F
\filldraw[fill=white, draw=black, line width=0.3mm] (1.7, -0.3) rectangle +(0.6, 0.6);
\filldraw[fill=white, draw=black, line width=0.3mm] (5.7, -0.3) rectangle +(0.6, 0.6);
\filldraw[fill=white, draw=black, line width=0.3mm] (9.7, -0.3) rectangle +(0.6, 0.6);
\filldraw[fill=white, draw=black, line width=0.3mm] (7.7, 1.7) rectangle +(0.6, 0.6);
%square odd cycle
\filldraw[fill=white, draw=black, line width=0.3mm] (-1.3, 1.7) rectangle +(0.6, 0.6);
\filldraw[fill=white, draw=black, line width=0.3mm] (-1.3, -2.3) rectangle +(0.6, 0.6);
\foreach \i in {1,3,5,7,9,11,13} 
{
\draw(v\i)[fill=black] circle(\vr);
}
\foreach \i in {2,4,6,8,10,12} 
{
\draw(v\i)[fill=white] circle(\vr);
}
% text
\node at (14.0,2.7) {$H_1$};
\end{scope}
\begin{scope}[xshift=0cm, yshift=-6cm] % second critical F-cactus H1 
% vertices defined
\coordinate(v1) at (-2.0,0.0);
\coordinate(v2) at (0.0,0.0);
\coordinate(v3) at (4.0,0.0);
\coordinate(v4) at (8.0,0.0);
\coordinate(v5) at (10.0,0.0);
\coordinate(v6) at (12.0,0.0);
\coordinate(v7) at (14.0,0.0);
\coordinate(v8) at (10.0,2.0);
\coordinate(v9) at (12.0,2.0);
\coordinate(v10) at (1, 2.0);
\coordinate(v11) at (3.0, 2.0);
\coordinate(v12) at (1.0,-2.0);
\coordinate(v13) at (3.0,-2.0);
\coordinate(v14) at (5, 2.0);
\coordinate(v15) at (7.0, 2.0);
%\coordinate(v16) at (5.0,-2.0);
%\coordinate(v17) at (7.0,-2.0);

% \edges		
\draw (v1) -- (v2); 
\draw (v4) -- (v7); 
\draw (v5) -- (v8) -- (v9);  
\draw (v2) -- (v10) -- (v11) -- (v3);
\draw (v2) -- (v12) -- (v13) -- (v3);
\draw (v3) -- (v14) -- (v15) -- (v4);
\draw (v3)  -- (v4);
%  vertices

% squares F
\filldraw[fill=white, draw=black, line width=0.3mm] (-0.3, -0.3) rectangle +(0.6, 0.6);
\filldraw[fill=white, draw=black, line width=0.3mm] (9.7, 1.7) rectangle +(0.6, 0.6);
\filldraw[fill=white, draw=black, line width=0.3mm] (11.7, -0.3) rectangle +(0.6, 0.6);
%square odd cycle
\filldraw[fill=white, draw=black, line width=0.3mm] (2.7, 1.7) rectangle +(0.6, 0.6);
\filldraw[fill=white, draw=black, line width=0.3mm] (2.7, -2.3) rectangle +(0.6, 0.6);
\filldraw[fill=white, draw=black, line width=0.3mm] (7.7, -0.3) rectangle +(0.6, 0.6);
\filldraw[fill=white, draw=black, line width=0.3mm] (4.7, 1.7) rectangle +(0.6, 0.6);
%\filldraw[fill=white, draw=black, line width=0.3mm] (4.7, -2.3) rectangle +(0.6, 0.6);
\foreach \i in {1,3,5,7,9,10,12,15} 
{
\draw(v\i)[fill=black] circle(\vr);
}
\foreach \i in {2,4,6,8,11,13,14} 
{
\draw(v\i)[fill=white] circle(\vr);
}
% text
\node at (14.0, 2.7) {$H_2$};
\end{scope} 

\end{tikzpicture}
\caption{Two atomic MBD critical cactus graphs $H_1$ and $H_2$, both obtained from the atomic MBD critical tree depicted at the bottom of Fig.~\ref{fig:atomic-critical-2x}.}
\label{fig:cactus}
\end{center}
\end{figure}

%%%%%%%%%%%%%%%%

\begin{theorem} \label{thm:cC}
    For every $H \in \cC$, the predominated graph $(H, V(H)\setminus X(H))$ is an atomic MBD critical graph.
\end{theorem}
\begin{proof}
The statement is true if $H$ is an isolated vertex, so we may assume that $H$ contains at least two vertices.
 In $(H, V(H)\setminus X(H))$, a vertex is predominated if and only if it is white. By Definitions~\ref{def:cS} and~\ref{def:cC}, every leaf in $H$ is black. Moreover, a cycle in $H$ is either an ``end-cycle" and its only vertex having a degree higher than $2$ is white, or the cycle has exactly two vertices of degree higher than $2$ and one of them is black, the other one is white.

Staller first claims a white cut-vertex $u$, breaking the graph into two components $H_1$ and $H_2$. Once Dominator replies in the component $H_i$, Staller's next move is a white cut-vertex from the unplayed component $H_{3-i}$. She continues playing according to this strategy, that is, she repeatedly plays a white cut-vertex in the component that remained unplayed so far. It is always possible to do that, except when the unplayed component  consists of a single black vertex $v$. But in that case all (white) neighbors of $v$ have already been played by Staller, and she wins the game by playing $v$ and, thus, claiming the closed neighborhood of an undominated vertex.  Note that the order of the considered component decreases with each move of Staller, and therefore, after a finite number of moves, only one black vertex $v$ remains.  

We next show that Dominator wins on $(H,D)$ if $D$ is obtained by adding a vertex $x \in X(H)$ to $V(H)\setminus X(H) $. If there is a matching $M$ in $H$ that covers all vertices except $x$, then Dominator can play according to the pairing strategy. That is, when Staller claims a vertex $u$ and $uu' \in M$, then Dominator's reply is $u'$ (and if $u'$ is already claimed, or Staller just claimed $x$, then Dominator claims an arbitrary vertex). This way, Dominator dominates every vertex, except possibly the predominated vertex $x$, and wins.

We denote by $M(H,x)$ a matching in $H$ that covers every vertex except $x$, see Fig.~\ref{fig:matchings}. To prove the existence of such a matching, we consider the following three cases. 

(i) If $H$ is a subdivided tree, i.e.~$H \in \cS$, we may get such a matching $M(H,x)$ by rooting the tree in $x$ and pairing every black vertex with its (white) parent. As the white vertices are the subdivision vertices, each of them has only one child. This way, the obtained matching that covers all vertices except $x$. 

(ii) If $H$ is an $F$-cactus for an $F \in \cS$ and $x \in X(F)$, we start with the matching $M(F, x)$ in $F$. Suppose now that, in the process of building $H$ from $F$, an edge $uv \in E(F)$ receives a double-odd replacement $uw_1 \dots w_k v$ and $uz_1 \dots z_\ell v$. Note that both $k$ and $\ell$ are even numbers (possibly equal to zero). If $uv \in M(F,x)$, we take the edges $uw_1, w_2w_3, \dots, w_kv$ and $z_1z_2, \dots, z_{\ell-1}z_\ell$ instead of $uv$. If $uv \notin M(F,x)$, we add the edges $w_1w_2, \dots, w_{k-1}w_k$ and $z_1z_2, \dots, z_{\ell-1}z_\ell$ to the matching. This way, we iteratively obtain a matching $M(H,x)$ in $H$ that omits only the vertex $x$. For an example of the described construction see the matching of $H_1$ in Fig.~\ref{fig:matchings}.

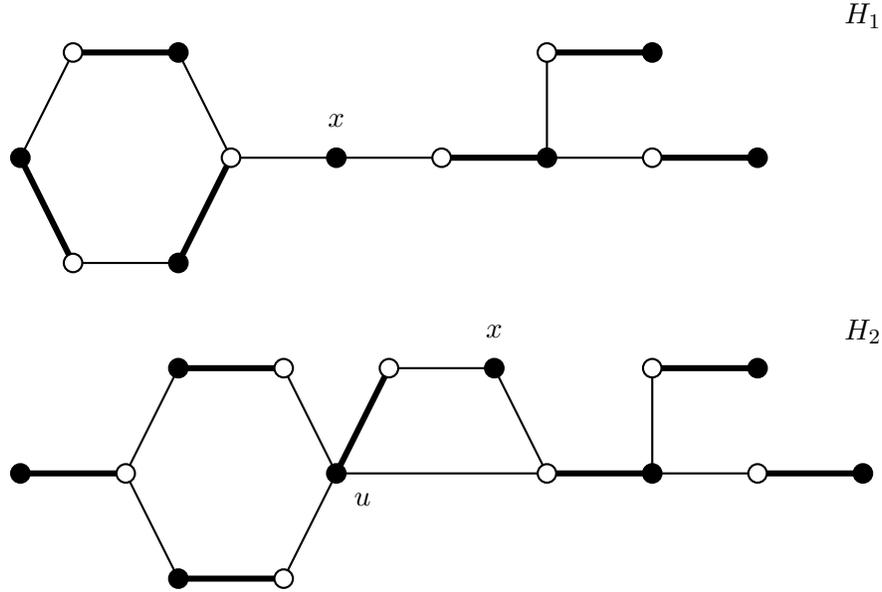
\begin{figure}[ht!]
\begin{center}
\begin{tikzpicture}[scale=0.7,style=thick,x=1cm,y=1cm]
\def\vr{5pt}

\begin{scope}[xshift=0cm, yshift=0cm] % first matching
% vertices defined
\coordinate(v1) at (-2.0,0.0);
\coordinate(v2) at (2.0,0.0);
\coordinate(v3) at (4.0,0.0);
\coordinate(v4) at (6.0,0.0);
\coordinate(v5) at (8.0,0.0);
\coordinate(v6) at (10.0,0.0);
\coordinate(v7) at (12.0,0.0);
\coordinate(v8) at (8.0,2.0);
\coordinate(v9) at (10.0,2.0);
\coordinate(v10) at (-1, 2.0);
\coordinate(v11) at (1.0, 2.0);
\coordinate(v12) at (-1.0,-2.0);
\coordinate(v13) at (1.0,-2.0);
% \edges		
\draw (v2) -- (v7);  
\draw (v5) -- (v8) -- (v9);  
\draw (v1) -- (v10) -- (v11) -- (v2);
\draw (v1) -- (v12) -- (v13) -- (v2);
\draw[line width=0.8mm] (v4) -- (v5);
\draw[line width=0.8mm] (v6) -- (v7);
\draw[line width=0.8mm] (v8) -- (v9);
\draw[line width=0.8mm] (v2) -- (v13);
\draw[line width=0.8mm] (v1) -- (v12);
\draw[line width=0.8mm] (v10) -- (v11);
%  vertices

% squares F
\foreach \i in {1,3,5,7,9,11,13} 
{
\draw(v\i)[fill=black] circle(\vr);
}
\foreach \i in {2,4,6,8,10,12} 
{
\draw(v\i)[fill=white] circle(\vr);
}
% text
\node at (4,0.7) {$x$};
\node at (14.0,2.7) {$H_1$};
\end{scope}
\begin{scope}[xshift=0cm, yshift=-6cm] % second matching
% vertices defined
\coordinate(v1) at (-2.0,0.0);
\coordinate(v2) at (0.0,0.0);
\coordinate(v3) at (4.0,0.0);
\coordinate(v4) at (8.0,0.0);
\coordinate(v5) at (10.0,0.0);
\coordinate(v6) at (12.0,0.0);
\coordinate(v7) at (14.0,0.0);
\coordinate(v8) at (10.0,2.0);
\coordinate(v9) at (12.0,2.0);
\coordinate(v10) at (1, 2.0);
\coordinate(v11) at (3.0, 2.0);
\coordinate(v12) at (1.0,-2.0);
\coordinate(v13) at (3.0,-2.0);
\coordinate(v14) at (5, 2.0);
\coordinate(v15) at (7.0, 2.0);
%\coordinate(v16) at (5.0,-2.0);
%\coordinate(v17) at (7.0,-2.0);

% \edges		
\draw (v1) -- (v2); 
\draw (v4) -- (v7); 
\draw (v5) -- (v8) -- (v9);  
\draw (v2) -- (v10) -- (v11) -- (v3);
\draw (v2) -- (v12) -- (v13) -- (v3);
\draw (v3) -- (v14) -- (v15) -- (v4);
\draw (v3)  -- (v4);
\draw[line width=0.8mm] (v1) -- (v2);
\draw[line width=0.8mm] (v4) -- (v5);
\draw[line width=0.8mm] (v6) -- (v7);
\draw[line width=0.8mm] (v8) -- (v9);
\draw[line width=0.8mm] (v10) -- (v11);
\draw[line width=0.8mm] (v12) -- (v13);
\draw[line width=0.8mm] (v3) -- (v14);
%  vertices
\foreach \i in {1,3,5,7,9,10,12,15} 
{
\draw(v\i)[fill=black] circle(\vr);
}
\foreach \i in {2,4,6,8,11,13,14} 
{
\draw(v\i)[fill=white] circle(\vr);
}
% text
\node at (7,2.7) {$x$};
\node at (14.0,2.7) {$H_2$};
\node at (4.5,-0.5) {$u$};
\end{scope} 

\end{tikzpicture}
\caption{Matchings $M(H_1,x)$ and $M(H_2,y)$, as defined in the proof of Theorem~\ref{thm:cC}.}
\label{fig:matchings}
\end{center}
\end{figure}

(iii) If $H$ is an $F$-cactus for an $F \in \cS$, but $x\in X(H) \setminus X(F)$, it is a black vertex from a cycle $C$. Let $u$ be the black vertex of degree at least $3$ from $C$ (or any black vertex of $C$, if $C$ is an end-cycle). We again start from the matching $M(F, u)$ on the tree $F$. Using $M(F,u)$, analogously to the approach from case~(ii), we can construct a matching covering all vertices of $H$ except $u$. Now it remains to locate a path between $x$ and $u$ with all internal vertices of degree two and swap the matched vertices along that path, obtaining the matching $M(H,x)$ that covers all vertices of $H$ except $x$. For an example of the described construction see the matching of $H_2$ in Fig.~\ref{fig:matchings}.

We have proved that $(H, V(H)\setminus X(H))$ is an MBD critical graph for every $H \in \cC$. It is clear that $H$ is connected and, as $V(H)\setminus X(H)$ is a partite class in $H$, there are no edges between predominated vertices. Thus, $(H, V(H)\setminus X(H))$ is an atomic MBD critical graph. 
\qed        
\end{proof}

Theorem~\ref{thm:cC} and Proposition~\ref{prop:MBDcritical} readily imply the following statement.
\begin{corollary} \label{cor:cactus-substr}
   If  $H \in \cC$ is a substructure in a graph $G$, then $(G, V(G) \setminus X(H))$ is an MBD critical graph. 
\end{corollary}

We conclude the section with the following.

\begin{problem}
Is it true that a predominated cactus graph $(G,D)$ is MBD critical if and only if it contains a cactus-substructure $H \in \cC$ and  $D= V(G) \setminus X(H)$?
\end{problem}

We remark that this problem is equivalent to the question whether $\cC$ contains all atomic MBD critical cactus graphs.

\iffalse
Example with an odd cycle: The bull graph $B$ is obtained from a triangle $v_1v_2v_3$ by adding the pendant edges $v_1u_1$ and $v_2u_2$. The predominated graph $(B, \{v_1,v_2\})$ is an MBD critical cactus, but it is not an atomic one as $v_1$ and $v_2$ are adjacent vertices. The cactus-substructure $H \in \cC$  in $(B, \{v_1,v_2\})$ is the $P_5$-subgraph  $u_1v_1v_3v_2u_2$ together with the set $D=\{v_1, v_2\}$ of predominated vertices.
\fi

%%%%%%%%%%%%%%%%%%%%%%%%%%%%%%%%%%%%%%%%%%%%%%%%%%%%%

\section{Concluding discussions} 
\label{sec:families-MB-critical-hgs}

In this section, we explore additional aspects of criticality that complement our previous investigations, draw conclusions from earlier results, and highlight promising directions for future research.  

%%%%%%%%%%%%%%%%%%%%%%%%%%%%%%%%%%%%%%%%
\subsection{More atomic MBD critical graphs}

In Definition~\ref{def:cC}, we defined a double-odd replacement of an edge $e$. By this operator, we obtained a set $\cC$ of atomic MBD critical cactus graphs.
In an analogous way, we define the $k$-odd replacement of an edge $e=xy$ of a graph $G$ and obtain a wider class of atomic MBD critical graphs.
\begin{definition}
     \label{def:cA}
   For an integer $k\ge 2$, the \emph{$k$-odd replacement} of an edge $e=xy$ in a graph $G$ is the replacement of the edge $e$ with $k$ internally vertex-disjoint $x,y$-paths of arbitrary odd lengths.
   The class $\cA$ is the minimal family that satisfies conditions $(i)$ and $(ii)$:
   \begin{itemize}
       \item[$(i)$] $\cS \subseteq \cA$.
       \item[$(ii)$] If $A \in \cA$ and $A'$ is obtained from $A$ by a $k$-odd replacement of an edge $e \in E(A)$, then $A' \in \cA$. The set $X(A')$ is the partition class of the bipartite graph $A'$ such that $X(A) \subseteq X(A')$.
      \end{itemize}
\end{definition}
The following theorem can be proved along the lines of the proof of Theorem~\ref{thm:cC}. 
\begin{theorem} \label{thm:cA}
    For every $A \in \cA$, the predominated graph $(A, V(A) \setminus X(A))$ is an atomic MBD critical graph.
\end{theorem}

We are curious whether the same approach could be extended further to cover wider families of graphs.

%%%%%%%%%%%%%%%%%%%%%%%%%%%%%%%%%%%%%
\subsection{Criticality for Dominator}

The main part of this paper studies MBD critical predominated graphs from Staller's point of view. However, the analogous situation for Dominator's win may also be interesting to explore.
\begin{definition} \label{def:D-critical}
A predominated graph $(G,D)$ is \emph{MBD Dominator-critical}, if $D\neq \emptyset$, Dominator wins in the Staller-start game on $(G,D)$, and Staller wins on $(G, D\setminus \{v\})$ for every $v \in D$.   
\end{definition}
For instance, take a star with four leaves $\ell_1, \dots, \ell_4$ and subdivide each edge twice to obtain the graph $R$. Let $D=\{\ell_1, \ell_2, \ell_3\}.$ By Theorem~\ref{thm:tree-St-wins}, Dominator wins on $(R, D)$, but every proper subset $D'$ of $D$ allows Staller to win on $(R, D')$. Hence, $(R,D)$ is MBD Dominator-critical.

Another example is depicted in Fig.~\ref{fig:D-critical}. 
%(Its criticality can be easily checked by applying Proposition~\ref{prop:D-crit-tree}.)
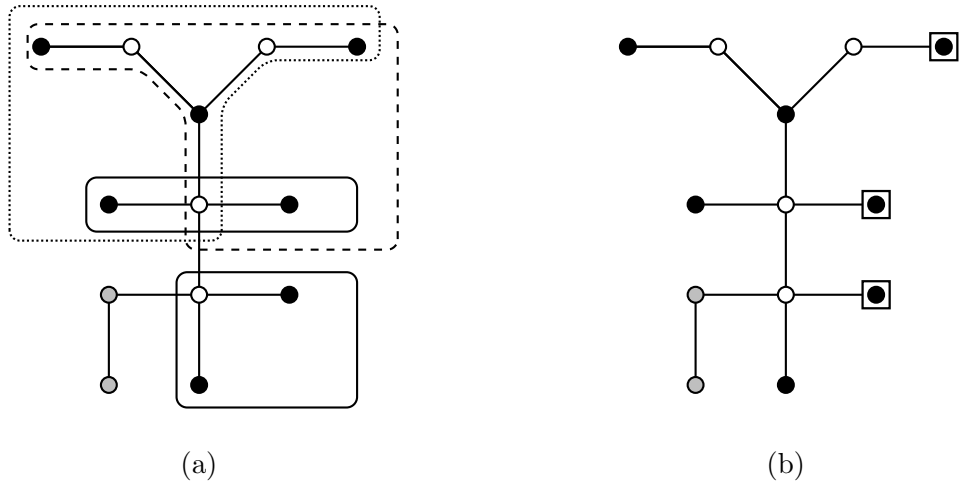
\begin{figure}[ht!]
\begin{center}
\begin{tikzpicture}[scale=0.6,style=thick,x=1cm,y=1cm]
\def\vr{5pt}
\begin{scope}[xshift=0cm, yshift=0cm]
% vertices defined
\coordinate(v1) at (2.0,0.0);
\coordinate(v2) at (4.0,0.0);
\coordinate(v4) at (2.0,2.0);
\coordinate(v5) at (4.0,2.0);
\coordinate(v6) at (6.0,2.0);
\coordinate(v7) at (2.0,4.0);
\coordinate(v8) at (4.0,4.0);
\coordinate(v9) at (6.0,4.0);
\coordinate(v10) at (4.0,6.0);
\coordinate(v11) at (0.5,7.5);
\coordinate(v12) at (2.5,7.5);
\coordinate(v13) at (5.5,7.5);
\coordinate(v14) at (7.5,7.5);
% \edges		
\draw (v2) -- (v5) -- (v4);  
\draw (v6) -- (v5) -- (v8) -- (v7);  
\draw (v9) -- (v8) -- (v10) -- (v12) -- (v11);  
\draw (v10) -- (v13) -- (v14);  
\draw (v10) -- (v12) -- (v11);  
\draw (v1) -- (v4);
\draw[rounded corners, densely dotted](-0.2,3.2)--(-0.2,8.4)--(8.0,8.4)--(8.0,7.2)--(5.5,7.2)--(4.5,6.2)--(4.5,3.2)--cycle;
\draw[rounded corners,dashed](8.4,3.0)--(8.4,8.0)--(0.2,8.0)--(0.2,7.0)--(2.7,7.0)--(3.7,6.0)--(3.7,3.0)--cycle;
\draw[rounded corners](1.5,3.4)--(1.5,4.6)--(7.5,4.6)--(7.5,3.4)--cycle;
\draw[rounded corners](3.5,-0.5)--(3.5,2.5)--(7.5,2.5)--(7.5,-0.5)--cycle;
% \draw (v12) .. controls (-4,3.3) and (-2,3.3) .. (v14);
% \draw plot [smooth cycle] coordinates {(-5.0,3.0) (-5.3,4.5) (-4.7,4.5)};
%  vertices
\foreach \i in {2,6,7,9,10,11,14} 
{
\draw(v\i)[fill=black] circle(\vr);
}
\foreach \i in {5,8,12,13} 
{
\draw(v\i)[fill=white] circle(\vr);
}
\foreach \i in {1,4} 
{
\draw(v\i)[fill=lightgray] circle(\vr);
}
% \draw(v\i)[fill=lightgray] circle(\vr);
%\path [fill=white] (0.1,0.1) rectangle (0.5,0.5);
% \filldraw[fill=white, draw=black, line width=0.3mm] (1.8, -0.2) rectangle +(0.4, 0.4);
% \filldraw[fill=white, draw=black, line width=0.3mm] (1.8, 1.8) rectangle +(0.4, 0.4);
% text
%\node at (1.0,6.0) {$F_1$};
%\node at (7.0,6.0) {$F_2$};
%\node at (7.0,4.0) {$F_3$};
\node at (4.0,-1.8) {(a)};
\end{scope}
%%%%%%%%%%%%%%%%%%%%%%%%%%
\begin{scope}[xshift=13cm, yshift=0cm]
% vertices defined
\coordinate(v1) at (2.0,0.0);
\coordinate(v2) at (4.0,0.0);
\coordinate(v4) at (2.0,2.0);
\coordinate(v5) at (4.0,2.0);
\coordinate(v6) at (6.0,2.0);
\coordinate(v7) at (2.0,4.0);
\coordinate(v8) at (4.0,4.0);
\coordinate(v9) at (6.0,4.0);
\coordinate(v10) at (4.0,6.0);
\coordinate(v11) at (0.5,7.5);
\coordinate(v12) at (2.5,7.5);
\coordinate(v13) at (5.5,7.5);
\coordinate(v14) at (7.5,7.5);
% \edges		
\draw (v2) -- (v5) -- (v4);  
\draw (v6) -- (v5) -- (v8) -- (v7);  
\draw (v9) -- (v8) -- (v10) -- (v12) -- (v11);  
\draw (v10) -- (v13) -- (v14);  
\draw (v10) -- (v12) -- (v11);  
\draw (v1) -- (v4);
%\draw[rounded corners, densely dotted](-0.2,3.2)--(-0.2,8.4)--(8.0,8.4)--(8.0,7.2)--(5.5,7.2)--(4.5,6.2)--(4.5,3.2)--cycle;
%\draw[rounded corners,dashed](8.4,3.0)--(8.4,8.0)--(0.2,8.0)--(0.2,7.0)--(2.7,7.0)--(3.7,6.0)--(3.7,3.0)--cycle;
%\draw[rounded corners](1.5,3.4)--(1.5,4.6)--(7.5,4.6)--(7.5,3.4)--cycle;
%\draw[rounded corners](3.5,-0.5)--(3.5,2.5)--(7.5,2.5)--(7.5,-0.5)--cycle;
% \draw (v12) .. controls (-4,3.3) and (-2,3.3) .. (v14);
% \draw plot [smooth cycle] coordinates {(-5.0,3.0) (-5.3,4.5) (-4.7,4.5)};
% squares F
\filldraw[fill=white, draw=black, line width=0.3mm] (5.7, 1.7) rectangle +(0.6, 0.6);
\filldraw[fill=white, draw=black, line width=0.3mm] (5.7, 3.7) rectangle +(0.6, 0.6);
\filldraw[fill=white, draw=black, line width=0.3mm] (7.2, 7.2) rectangle +(0.6, 0.6);
%  vertices
\foreach \i in {2,6,7,9,10,11,14} 
{
\draw(v\i)[fill=black] circle(\vr);
}
\foreach \i in {5,8,12,13} 
{
\draw(v\i)[fill=white] circle(\vr);
}
\foreach \i in {1,4} 
{
\draw(v\i)[fill=lightgray] circle(\vr);
}
% \draw(v\i)[fill=lightgray] circle(\vr);
%\path [fill=white] (0.1,0.1) rectangle (0.5,0.5);
% \filldraw[fill=white, draw=black, line width=0.3mm] (1.8, -0.2) rectangle +(0.4, 0.4);
% \filldraw[fill=white, draw=black, line width=0.3mm] (1.8, 1.8) rectangle +(0.4, 0.4);
% text
%\node at (1.0,6.0) {$F_1$};
%\node at (7.0,6.0) {$F_2$};
%\node at (7.0,4.0) {$F_3$};
%\node at (6.5,0.5) {$F_4$};
\node at (4.0,-1.8) {(b)};
\end{scope}

\end{tikzpicture}
\caption{The tree $T$ from Fig.~\ref{fig:substructures}. Picture $(a)$ emphasizes the four substructures of $T$, while $(b)$ shows three vertices marked by squares such that their predomination results in an MBD Dominator-critical tree.}  

\label{fig:D-critical}
\end{center}
\end{figure}

As we already noted, in Maker-Breaker games adding new winning sets is not a disadvantage for Maker, and removing winning sets is not a disadvantage for Breaker. Observation~\ref{obs:closed-neigh-hg-predomination} and Definition~\ref{def:D-critical} then directly imply the following properties. For part $(ii)$, we note that Dominator always wins if all vertices are predominated.
 \begin{observation} \enskip
     \label{obs:D-critical}
   \begin{itemize} 
       \item[$(i)$] If a predominated graph $(G,D) $ is MBD Dominator-critical and $D'\subseteq V(G)$ is a proper subset or a proper superset of $D$, then $(G,D')$ is not MBD Dominator-critical. 
       \item[$(ii)$] For a graph $G$, an MBD Dominator-critical predominated graph $(G,D)$ exists if and only if Staller wins the MBD game on $G$.
        \end{itemize}
 \end{observation}
Given a hypergraph $\cH=(V, E)$, a set $Y \subseteq V$ is a \emph{transversal} if $Y$ contains at least one vertex from every hyperedge. A transversal $Y$ is \emph{minimal} if no proper subset of it is a transversal in $\cH$.
If $T$ is a tree, we define the associated hypergraph $\cX_T$ on the vertex set of $T$ with the following edge set:
 \[ E(\cX_T)= \{X(S): \mbox{$S$ is a substructure in $T$}\};
 \]
 that is, every hyperedge of $\cX_T$ corresponds to the set of fixed-degree (black) vertices in a substructure in $T$. Any transversal in $\cX_T$ contains a black vertex from each substructure in $T$.
Applying this concept and Theorem~\ref{thm:tree-St-wins}, we obtain a characterization for MBD Dominator-critical trees.
\begin{proposition} \label{prop:D-critical-tree}
    Let $T$ be a tree, $D \subseteq V(T)$, and $\cX_T$ the hypergraph associated with $T$. The predominated graph $(T,D)$ is MBD Dominator-critical if and only if $\cX_T$ is not an empty hypergraph and $D$ is a minimal transversal in $\cX_T$.
\end{proposition}
\begin{proof}
   By Theorem~\ref{thm:cS}, Staller wins on a tree $T$ if and only if $\cX_T$ is not an empty hypergraph. Observation~\ref{obs:D-critical} (ii) then gives the same condition for the existence of a set $D$ such that the predominated graph $(T,D)$ is MBD Dominator-critical.
   So we may assume that $T$ contains at least one substructure and $\cX_T$ is not empty.

   If $(T,D)$ is MBD Dominator-critical, Dominator wins on $(T,D)$ and consequently, by Theorem~\ref{thm:tree-St-wins}, $D$ contains at least one black vertex from every substructure. Equivalently, $D$ is a transversal in $\cX_T$. The criticality also implies that Staller wins on $(T, D \setminus \{v\} )$ for every $v \in D$. 
   Again, by Theorem~\ref{thm:tree-St-wins}, the latter is equivalent to the property that no proper subset of $D$ is a transversal in $\cX_T$. It proves that $D$ is a minimal transversal. 
   
   The other direction of the statement can be proved similarly by referring to Theorem~\ref{thm:tree-St-wins}.  \qed
\end{proof}

%%%%%%%%%%%%%%%%%%%%%%%%%%%%%%%%%%%

\subsection{Families of Maker-Breaker critical hypergraphs}

In Section~\ref{sec:critical-hgs}, we defined (atomic) Maker-Breaker critical hypergraphs and made some basic observations. Now, having our results on MBD critical graphs in hand, we can add some further statements related to critical hypergraphs.

It is clear by definitions and Observation~\ref{obs:closed-neigh-hg-predomination} that if $(G,D)$ is an MBD critical (resp.,~atomic MBD critical) graph, then $\cN(G,D)$ is Maker-Breaker critical (resp.,~atomic Maker-Breaker critical) hypergraph. This interplay and Theorem~\ref{thm:critical-trees} give us a family of hypergraphs that are atomic Maker-Breaker critical. To have a more transparent description of this family, we rephrase~\cite[Definition 3.2]{bujtas-2023} which defines the family $\cS$. (Recall that, by \cite[Proposition 3.4]{bujtas-2023}, our Definition~\ref{def:cS} and ~\cite[Definition 3.2]{bujtas-2023} give the same family $\cS$ of substructures.) 

Let $G_1$ and $G_2$ be two disjoint graphs, and let $z$ be a vertex not in $V(G_1)\cup V(G_2)$. Let further $x_i\in V(G_i)$, $i\in [2]$. Then $Z = (G_1,x_1)+(G_2,x_2)$ is the graph with $V(Z) = V(G_1)\cup V(G_2)\cup\{z\}$ and $E(Z) = E(G_1)\cup E(G_2) \cup\{zx_1, zx_2\}$. 

\begin{definition}
$\cS$ is the minimal family of graphs that satisfies the following conditions: 
\begin{enumerate}
    \item[(i)] $P_1\in \cS$ with $X(P_1) = V(P_1)$. 
    \item[(ii)] If $S_1, S_2\in \cS$ and $z_i\in X(S_i)$, $i\in [2]$, then $Z = (S_1,z_1)+(S_2,z_2)\in \cS$ with $X(Z) = X(S_1)\cup X(S_2)$. 
\end{enumerate}
\end{definition}

Given two vertex disjoint hypergraphs $\cH_1$ and $\cH_2$ with specified edges $e_1$ from $\cH_1$ and $e_2$ from $\cH_2$, the hypergraph $\cH=(\cH_1, e_1) + (\cH_2, e_2)$ is defined on the vertex set $V(\cH_1) \cup V(\cH_2) \cup \{z\}$ (where $z$ is a new vertex) and with the edge set
\[ E(\cH)= E(\cH_1) \cup E(\cH_2) \cup \{e_1 \cup \{z\}, e_2 \cup \{z\} \} \setminus \{ e_1, e_2 \};
\]
that is, the new vertex $z$ is added to the specified edges to connect $\cH_1$ and $\cH_2$.

Further, $\cH^1$ denotes the $1$-uniform hypergraph with exactly one vertex and exactly one ($1$-element) edge.

\begin{definition} \label{def:cL}
    Let $\cL$ be the minimal family that satisfies the following conditions:
    \begin{itemize}
           % \item the ($1$-uniform) hypergraph with vertex set $V=\{v\}$ and edge set $E=\{ \{v\} \}$ belongs to $\cL$;
        \item[$(i)$] $\cH^1 \in \cL$.
        \item[$(ii)$] If $\cH_1 \in \cL$ and $\cH_2 \in \cL$, and $e_i$ is an arbitrary edge in $\cH_i$, for $i \in [2]$, then the hypergraph $(\cH_1, e_1) + (\cH_2, e_2)$ also belongs to $\cL$.
    \end{itemize}
\end{definition}

Theorem~\ref{thm:critical-trees} then implies the following statement.
\begin{proposition} \label{prop:cL}
   Every hypergraph in $\cL$ is an atomic Maker-Breaker critical hypergraph.
\end{proposition}

Theorem~\ref{thm:cA} generalizes Theorem~\ref{thm:cC} by stating that the predominated graph $(A, V(A) \setminus X(A))$ is atomic MBD critical for all $A \in \cA$. From this theorem and Observation~\ref{obs:closed-neigh-hg-predomination} we can derive the following statement.
\begin{proposition} \label{prop:cactus-hg}
   For every  $A \in \cA$, the hypergraph $\cN(A, V(A) \setminus X(A))$ is an atomic Maker-Breaker critical hypergraph.
\end{proposition}

%%%%%%%%%%%%%%%%%%%%%%%%%%%%%%%%%%%%%%%%%%%%%%%%%%%%%%%
\section*{Acknowledgments}
%%%%%%%%%%%%%%%%%%%%%%%%%%%%%%%%%%%%%%%%%%%%%%%%%%%%%%%

Csilla Bujt\'as, Pakanun Dokyeesun, and Sandi Klav\v zar were supported by the Slovenian Research and Innovation Agency (ARIS) under the grants P1-0297, N1-0355, and N1-0285. 
Milo\v{s} Stojakovi\'{c} was partly supported by Ministry of Science, Technological Development and Innovation of Republic of Serbia (Grants 451-03-137/2025-03/200125 \& 451-03-136/2025-03/200125), and partly supported by Provincial Secretariat for Higher Education and Scientific Research, Province of Vojvodina (Grant No.~142-451-2686/2021).

%%%%%%%%%%%%%%%%%%%%%%%%%%%%%%%%%%%%%%%%%%%%%%

\end{document}